\documentclass[12pt]{article}

\DeclareMathAlphabet{\mathpzc}{OT1}{pzc}{m}{it}

\usepackage{amsmath}
\usepackage{amssymb}
\usepackage{amsfonts}
\usepackage{amsthm}
\usepackage{mathrsfs}
\usepackage{enumerate}
\usepackage[pdftex]{color, graphicx}

\newcommand{\cI}{{\mathcal I}}
\newcommand{\cS}{{\mathcal S}}

\newcommand{\pa}{\parallel}

\newcommand{\li}{l_\infty}

\allowdisplaybreaks

\newtheorem{thm}{Theorem}[section]
\newtheorem{lem}[thm]{Lemma}

\newtheorem{cor}[thm]{Corollary}

\begin{document}

\renewcommand{\thefootnote}{\arabic{footnote}}
 	
\title{Synthetic foundations of cevian geometry, I:\\Fixed points of affine maps in triangle geometry}

\author{\renewcommand{\thefootnote}{\arabic{footnote}}
Igor Minevich and Patrick Morton\footnotemark[1]}
\footnotetext[1]{The authors were supported by an Honors Research Fellowship from the IUPUI Honors Program from 2007-2009, during the period in which this work was completed.}
\maketitle

\begin{section}{Introduction}
\end{section}

Let the cyclocevian mapping for triangle $ABC$ be denoted by $\phi$, so that by definition the traces of $P$ and $\phi(P)$ on the extended sides of triangle $ABC$ lie on a common circle.  By early 2004, the second author had independently discovered and proved the fact that
$$\phi=\iota \circ  \gamma' \circ \iota \eqno{(1)}$$ \smallskip
where $\iota$ is the isotomic mapping for triangle $ABC$ and $\gamma'$ is the isogonal mapping for the anti-complementary triangle of $ABC$.  The proof made extensive use of computer-assisted algebra and absolute barycentric coordinates.  The coordinates of $\phi(P)$ are 8th degree rational functions in the barycentric coordinates $(u, v, w)$ of $P$, but when $\phi$ is conjugated by the isotomic mapping there is a remarkable drop in degree:  $\iota \circ \phi \circ \iota$ becomes a 2nd degree rational function in $(u, v, w)$, which turns out to be the same as $\gamma'(P)$.  The formulas that occur in this proof can be given a nice form, but are difficult to verify by hand.  Since this discovery we have been looking for a synthetic proof for (1). \medskip

Let $K$ denote the complement mapping with respect to triangle $ABC$, so that $K$ is the affine mapping which maps $P$ to a point $K(P)$ on line $PG$ ($G$ the centroid) for which the signed length $K(P)G =\frac{1}{2} GP$.  During the previous year, unbeknownst to us, Grinberg, in Hyacinthos 6423 [gr1], had announced the equivalent formula
$$\phi = \iota \circ K^{-1} \circ \gamma \circ K \circ \iota,\eqno{(2)}$$
($\gamma$ is now the isogonal map for $ABC$ itself) which he derived using the concept of the {\it isotomcomplement} of a point $P$ with respect to $ABC$ (also called the inferior of the isotomic conjugate of $P$, see [y2]).  This is the point
$$Q = K \circ \iota(P)$$ \smallskip
Grinberg derived his formula with the help of homogeneous barycentric coordinates, and noted that he had found synthetic arguments for all but one of the main facts he had used in his proof of (2), which is Theorem 1 in his message [gr1], and is given here as Theorem 2.4 in Section 2 below.  This theorem is also contained in Paul Yiu's message [y2] in slightly disguised form. \medskip

The starting point for this paper is to show how this theorem of Grinberg and Yiu can be proved synthetically, thus filling in the synthetic gap in Grinberg's argument for (2).  The key step is the Midpoint Perspectivity Theorem (Theorem 2.3).  For completeness we give the proof of Grinberg's formula (2) using the cross-ratio in Theorem 2.7.  \medskip

In the process of finding this proof, we discovered that we could synthetically prove many of the other facts concerning the isotomcomplement that have been noted in the Hyacinthos messages.  For example, we prove synthetically Ehrmann's observation [ehr] that the point $Q=K \circ \iota(P)$ is a fixed point of the affine mapping $T_P$ which maps triangle $ABC$ to the cevian triangle $DEF$ of $P$ with respect to $ABC$ (see Theorem 3.2).  We show in addition that $Q$ is the only fixed point of $T_P$ in the finite plane (under suitable hypotheses; see Theorem 3.11) and that $T_P \circ K(P)=P$ (Theorem 3.6). Finally, we show that if $P'=\iota(P)$ is the isotomic conjugate of $P$ with respect to $ABC$, then the affine mapping $T_P \circ T_{P'}$ has a fixed point which is the $P$-ceva conjugate of $Q$, i.e. the perspector of the cevian triangle of $P$ and the anti-cevian triangle of the isotomcomplement $Q$ of $P$ (both with respect to $ABC$).  With this notation the anti-cevian triangle of $Q$ with respect to $ABC$ turns out to be simply $T_{P'}^{-1}(ABC)$.  Again, the set of points for which $P'=\iota(P)$ is on the line at infinity, which is by definition the Steiner circumellipse, can be characterized in terms of the mappings $T_P$ and $T_{P'}$ as the set of ordinary points $P$ for which $T_P \circ T_{P'} = K^{-1}$ equals the inverse of the complement map for triangle $ABC$.  \medskip

What results is a completely synthetic treatment of many new results in the theory of cevian triangles.  This turns out to be an extended and entertaining exercise in classical projective geometry.  Moreover, our development shows that there are many interesting connections between affine maps and the cevian geometry of a triangle. In particular, many important points in triangle geometry can be synthetically characterized as fixed points of specific affine maps. \medskip

In further papers we will explore this connection more fully.  In Part II of this paper [mm] we will study the conic $\mathcal{C}_P=ABCPQ$ on the five points $A,B,C,P,Q$, and the center $Z$ of $\mathcal{C}_P$, i.e., the pole of the line at infinity with respect to $\mathcal{C}_P$.  We will show (synthetically) that if $P$ does not lie on a median of $ABC$ or on $\iota(\li)$, then $Z$ is a fixed point of the affine mapping $\lambda=T_{P'} \circ T_P^{-1}$, and is the unique fixed point if $\mathcal{C}_P$ is a parabola or ellipse.  When $\mathcal{C}_P$ is a hyperbola, there are two more fixed points at infinity.  We will also prove that the point $Z$ is the intersection $GV \cdot T_P(GV)$, where $V$ is the midpoint of the segment joining $K^{-1}(P)$ and $K^{-1}(P')$ and $G$ is the centroid. \medskip

In Part III we will study the generalized orthocenter $H$ of $P$, defined to be the intersection of the lines through the vertices $A, B, C$ which are parallel, respectively, to the lines $QD, QE, QF$.  We will prove that $H$ lies on the conic $\mathcal{C}_P$, along with $P', Q'$, and $H'$, which are the $Q$- and $H$-points defined for the point $P'$ in place of $P$.  We also prove the affine formula
$$H=K^{-1} \circ T_{P'}^{-1} \circ K(Q),$$
and show that the map $\lambda$ defined above satisfies $\lambda(H)=Q$. \medskip

In the final paper of this series (Part IV) we will prove synthetically that $Z$ is the generalized Feuerbach point, the point where the nine point conic $\mathcal{N}_H$ of the quadrangle $ABCH$ is tangent to the inscribed conic $\mathcal{I}$ of $ABC$, the conic which is tangent to the sides at the points $D, E, F$.  The proof proceeds by showing that the map
$$\Phi_P=T_P \circ K^{-1} \circ T_{P'} \circ K^{-1}$$
is a homothety or translation taking the conic $\mathcal{N}_H$ to the inconic $\mathcal{I}$ and fixing the point $Z$.  (See [mo].)  In this part we will also prove synthetically that the isogonal conjugate $\gamma(H)$ of the point $H$ is the perspector of the tangential triangle of $ABC$ (the triangle tangent to the circumcircle at the vertices) and the circumcevian triangle of the point $\gamma(Q)$.  (The circumcevian triangle of a point $R$ with respect to $ABC$ is the triangle $A'B'C'$ inscribed in the circumcircle of $ABC$ with the property that $\{A,R, A'\}$, $\{B, R, B'\}$, and $\{C, R, C'\}$ are triples of collinear points.) \medskip

This paper forms the foundation for all the results that follow in the sequel. \medskip

 {\it Notation.}  We use the results and notation of Coxeter's book [cox], which gives a synthetic development of projective geometry.  See also [ev] for many of the elementary geometrical results and concepts that we use, including directed distance; and [wo] or [y1] or [y3] for definitions of terms in triangle geometry. \medskip

More specifically, we use the following notation.  If $P$ is an ordinary point, and $ABC$ is an ordinary triangle, we have the following cevian triangles:\\

$D_0E_0F_0$, the cevian triangle of the centroid $G$, i.e., the medial triangle of $ABC$, where $D_0 = AG \cdot BC, E_0 = BG \cdot AC, F_0 = CG \cdot AB;$\\

$DEF=D_1E_1F_1$, the cevian triangle for $P$ with respect to $ABC$; where $D=D_1=AP \cdot BC$, etc.;\\

$D_2E_2F_2$, for $Q=K(P')=K \circ \iota(P)$; where $D_2 = AQ \cdot BC$, etc;\\

$D_3E_3F_3$ for $P' = \iota(P)$; where $D_3=AP' \cdot BC$, etc.;\\

$D_4E_4F_4$ for $Q' = K \circ \iota(P')=K(P)$; where $D_4=AQ' \cdot BC$, etc.;\\

$D_5E_5F_5$ for $X$, the fixed point of $\mathcal{S}=T_P \circ T_{P'}$; see Theorem 3.5.\\

\noindent Also, we set $A_i = T_P(D_i), B_i = T_P(E_i), C_i = T_P(F_i)$, for $0 \le i \le 5$.  We will occasionally replace $P$ by the point $P'$, and put primes on the above listed points, to indicate that they correspond to the point $P'$, so that, for example, $D_1'=D_3$, $A_i'=T_{P'}(D_i')$, etc.  \bigskip

\begin{section}{The Grinberg-Yiu theorem.}

In this section we explore some basic properties of the isotomcomplement of a point.  We always take the vertices of the triangle $ABC$ to be ordinary.  We assume throughout this section that $P$ does not lie on the extended sides of triangle $ABC$ or its anti-complementary triangle, so that the vertices of its cevian triangle $DEF$ are always ordinary points.  Usually we will assume $P$ is also ordinary, but the isotomic point $\iota(P)=P'$ of $P$ may be infinite, if $P$ lies on $\iota(\li)$, the Steiner circumellipse for $ABC$.  (Cf. Theorem 3.13. $\li$ is the line at infinity.)  Note, however, that most of our proofs work when $P$ is infinite (in that case $P'$ is ordinary).  As in the introduction, $K$ denotes the complement map with respect to $ABC$ and $Q=K(P')=K \circ \iota (P)$ is the isotomcomplement of $P$. \medskip

Further, we let $D_0=K(A), E_0=K(B)$, and $F_0=K(C)$ denote the midpoints of sides $a=BC, b=AC$, and $c=AB$, repsectively.

\begin{thm}[{Theorem 3 in [gr1]}] Let $ABC$ be a triangle and $D, E, F$ the traces of point $P$ on the sides opposite $A, B$, and $C$. Let $D_0, E_0, F_0$ be the midpoints of the sides opposite $A, B, C$, and let $M_d, M_e, M_f$ be the midpoints of $AD, BE, CF$. Then $D_0M_d, E_0M_e, F_0M_f$ meet at the isotomcomplement $Q = K \circ \iota(P)$ of $P$.
\label{thm:1.1}
\end{thm}

\begin{proof}
(See Figure 1.)  We may assume $P \ne$ the centroid $G$ of $ABC$. Without loss of generality, assume $P$ is not on $AG$.  Let $D_3$ be the trace of $P'$ on side $BC$.  Then $D_0M_d$ is the midline of $\Delta DAD_3$ and is thus parallel to $AD_3 = AP'$. Draw $P'G$, and let $M$ be the intersection $D_0M_d \cdot P'G$. Since $D_0M_d$ is parallel to $AP'$, the triangles $GMD_0$ and $GP'A$ are similar; so $AG = 2 GD_0$ implies $GP' = 2 GM$. Thus, $M = K(P')=Q$ and $D_0M_d$ intersects $GP'$ at $Q$; similarly, so do $E_0M_e$ and $F_0M_f$. If $P$ lies on $BG$ it cannot lie on $CG$; in that case, $Q=D_0M_d \cdot F_0M_f$, and $P$ on $BG$ implies that $P', Q$, and $M_e$ also lie on $BG$, giving the assertion.  This argument applies as long as the point $M$ is ordinary. If $M = P'G\cdot \li$ is infinite, then $P'G \pa D_0M_d \pa P'A$ implies, since $P \ne G$, that $P'$ is infinite and $M = P' = K(P') = Q$. Since $BP'$ and $CP'$ are now parallel to $AP'$, the lines $D_0M_d, E_0M_e, F_0M_f$ meet at $P' = Q$.
\end{proof}

\begin{figure}[htbp]
\[\includegraphics[width=4.5in]{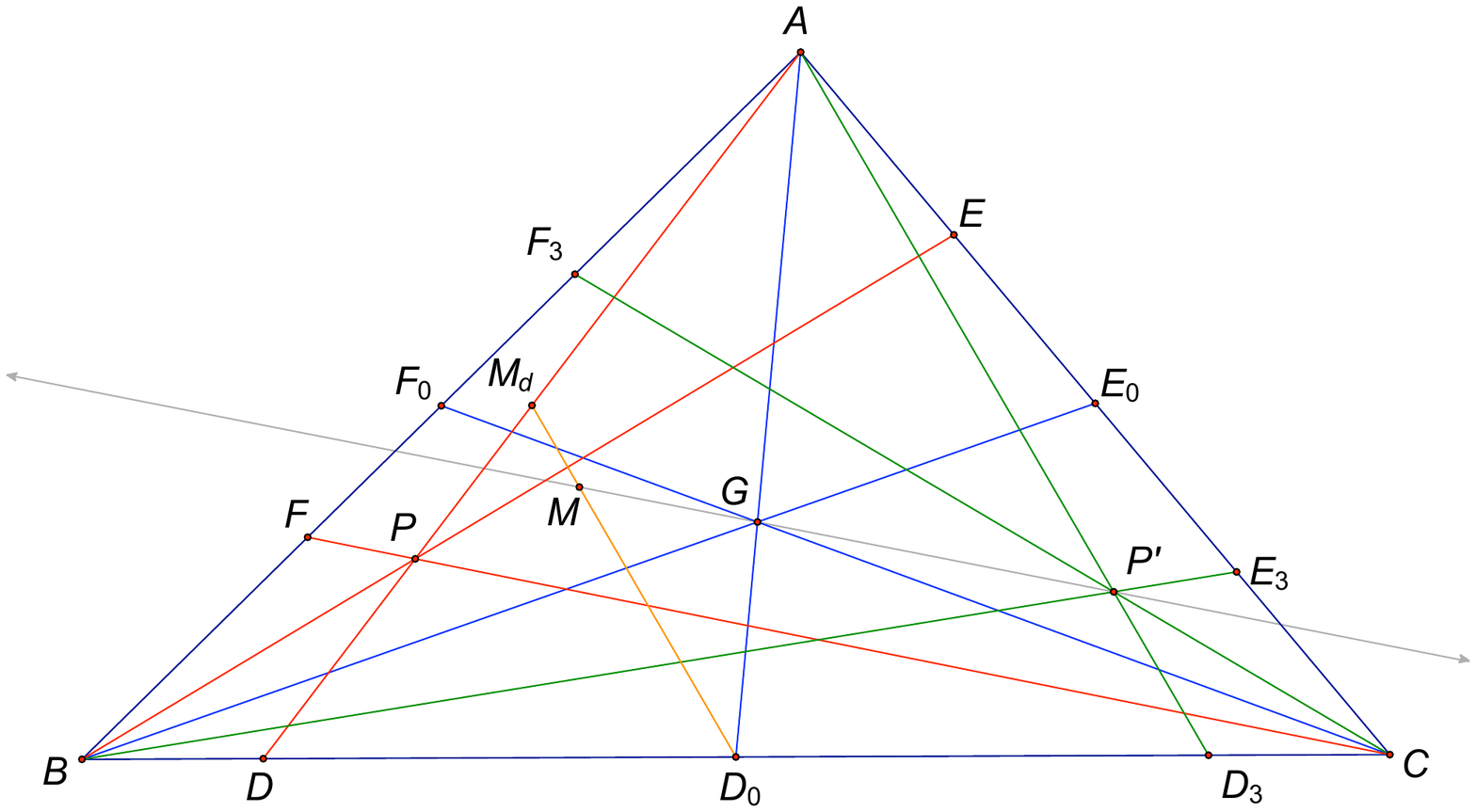}\]
\caption{Proof of Theorem 2.1}
\label{fig:1.1}
\end{figure}

\begin{cor}
$D_0M_d = D_0Q$ is parallel to $AP'$ and $K(D_3) = M_d$.
\label{cor:1.1}
\end{cor}

The above theorem is in Altshiller-Court [ac, p.165, Supp. Ex. 10], except for the identification of the intersection of the lines $D_0M_d, E_0M_e, F_0M_f$ as the isotomcomplement of $P$. \bigskip

\begin{figure}[hbtp]
\[\includegraphics[width=4.5in]{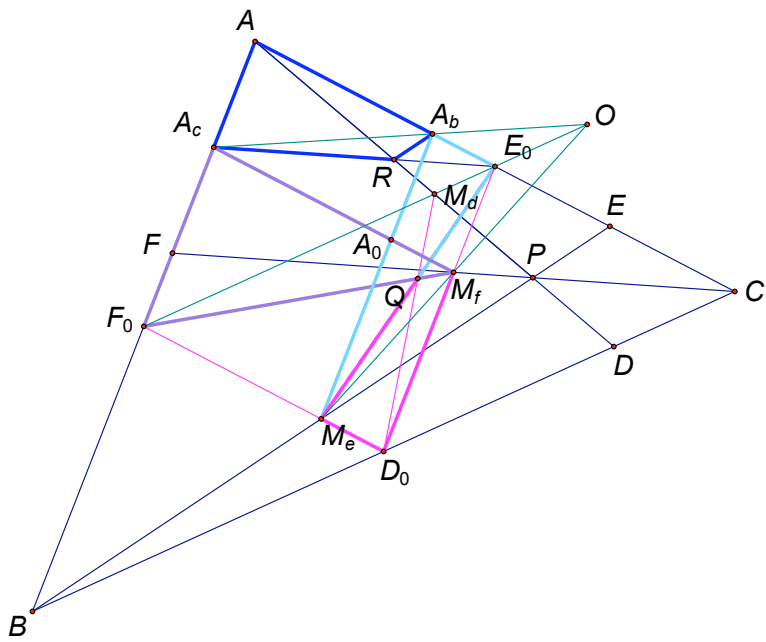}\]
\caption{Midpoint Perspectivity Theorem}
\label{fig:1.2}
\end{figure}

\begin{thm}[Midpoint Perspectivity Theorem]
In triangle $ABC$, let $E$ and $F$ be points on sides $AC$ and $AB$. Let:
\begin{center}
$E_0, F_0$ be the midpoints of $b = AC$ and $c = AB$,\\
$A_b, A_c$ be the midpoints of $AE$ and $AF$,\\
$M_e, M_f$ be the midpoints of $BE$ and $CF$.
\end{center}
Then the triangles $A_bE_0M_e$ and $A_cF_0M_f$ are perspective.
\label{thm:1.2}
\end{thm}

\begin{proof}
(See Figure 2.)  We want to show that $A_bA_c, E_0F_0, M_eM_f$ are concurrent at a point $O$. Using quadrangle $D_0M_eQM_f$, we have: $QM_f\cdot D_0M_e = F_0$ by Theorem \ref{thm:1.1} and the fact that $M_e$ lies on the midline $D_0F_0$. Similarly, $QM_e\cdot D_0M_f = E_0$ and $D_0Q\cdot E_0F_0 = M_d$. Defining $O = M_eM_f\cdot E_0F_0$, we obtain the harmonic relation $H(E_0F_0, M_dO)$, by the definition of a harmonic set. See \cite{cox}.\\
\\
Now we use quadrangle $AA_bRA_c$, where $R = A_bF_0\cdot A_cE_0$. Since $A_cE_0$ is the median of $AFC$, we have $A_cE_0\cdot AP =$ midpoint of $AP$ and similarly $A_bF_0\cdot AP =$ midpoint of $AP$, so $A_cE_0, A_bF_0, AP$ are concurrent at $R$. This implies that $AA_b\cdot A_cR = E_0$; $AA_c\cdot A_bR = F_0$; and $AR\cdot E_0F_0 = AP \cdot E_0F_0 = M_d$ in quadrangle $AA_bRA_c$. But this says that $A_bA_c\cdot E_0F_0$ is the harmonic conjugate of $M_d$ with respect to $E_0F_0$, which is $O$ by the first part of the argument. Therefore, $A_bA_c, E_0F_0, M_eM_f$ are concurrent at $O$.
\end{proof}

\noindent {\bf Remark.} Denoting the point $O$ in the proof of this theorem by $O_a$, there are similar points $O_b, O_c$ for the vertices $B, C$ (see Figure \ref{fig:1.2}). It is not hard to show that $O_a = K(E_3F_3\cdot BC), O_b = K(D_3F_3\cdot AC), O_c = K(D_3E_3\cdot AB)$, where $D_3, E_3, F_3$ are the traces of the point $P' = \iota(P)$. It follows that the points $O_a, O_b, O_c$ are collinear, and that the line through these points is the trilinear polar of the point $Q$ with respect to triangle $D_0E_0F_0$.

\begin{thm}[{Theorem 1.3 in [gr1]}] With $D, E, F$ as before, let $A_0$, $B_0$, $C_0$ be the midpoints of $EF, DF$, and $DE$, respectively. Then the lines $AA_0$, $BB_0$, $CC_0$ meet at the isotomcomplement $Q$ of $P$.
\label{thm:1.3}
\end{thm}

\begin{proof}
Using the notation of Theorem 2.1, $A_bM_e \cdot A_cM_f = A_0$ because $A_bM_e$ is a midline of $\Delta AEB$ and so passes through the midpoint of $EF$; similarly, $A_cM_f$ is a midline of $\Delta AFC$ and also passes through the midpoint of $FE$. Thus we have $A_bE_0\cdot A_cF_0 = A$; $A_bM_e\cdot A_cM_f = A_0$; and $E_0M_e\cdot F_0M_f = Q$, by Theorem \ref{thm:1.1}. By Theorem \ref{thm:1.2} and Desargues' theorem, these three points are collinear. Similarly, $B, B_0$, and $Q$ are collinear, and $C, C_0$, and $Q$ are collinear.
\end{proof}

Part of the statement of this proposition is in Altshiller-Court [ac, p. 165, Supp. Ex. 8]. \medskip

We now prove the following surprising theorem, which combines everything we have proved so far.

\begin{thm}[{Quadrilateral Half-turn Theorem}] If $Q'=K(P)$ is the isotomcomplement of $P'$, the complete quadrilaterals
$$(AP)(AQ)(D_0Q)(D_0A_0) \ \ and \ \ (D_0Q')(D_0A_0')(AP')(AQ')$$
are perspective by a Euclidean half-turn about the point $N_1=$ midpoint of $AD_0 =$ midpoint of $E_0F_0$.  In particular, corresponding sides in these quadrilaterals are parallel.
\end{thm}

\begin{figure}[htbp]
\[\includegraphics[width=5.5in]{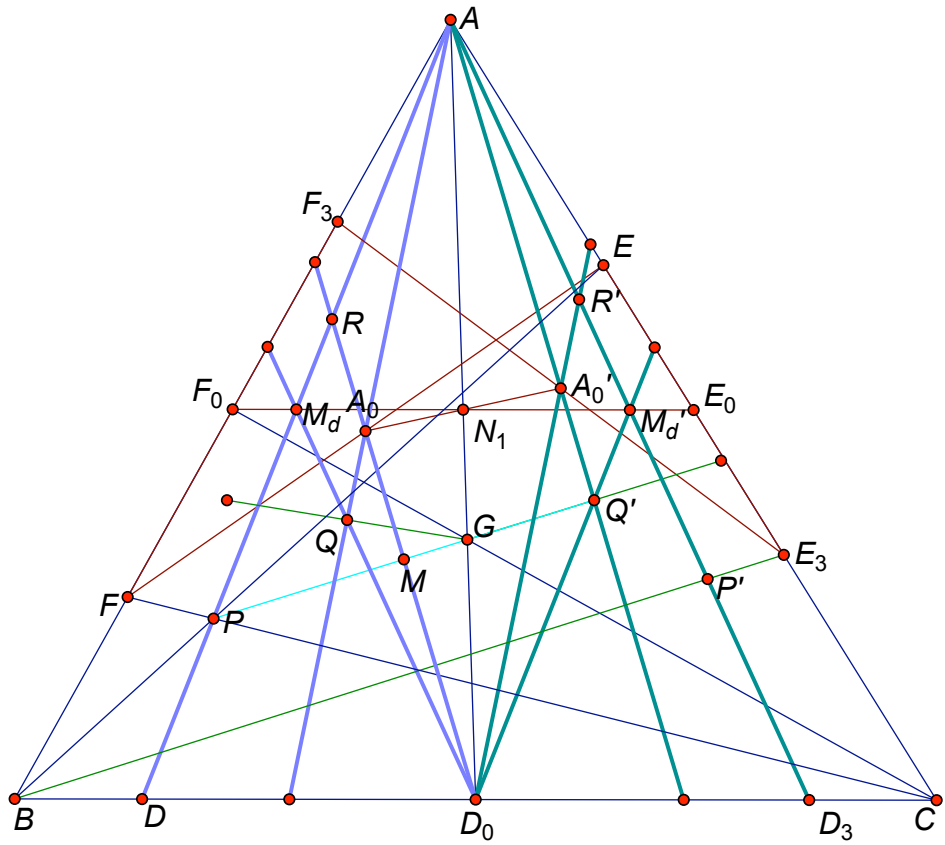}\]
\caption{Quadrilateral Half-turn Theorem}
\label{fig:1.3}
\end{figure}

\begin{proof}
(See Figure 3.) Let $R$ and $R'$ denote the midpoints of segments $AP$ and $AP'$, and $M_d$ and $M_d'$ the midpoints of segments $AD$ and $AD_3$, where $D_3=AP' \cdot BC$.  We first check that the vertices of the quadrilateral
$$\Lambda = (AP)(AQ)(D_0Q)(D_0A_0)$$
are $A, R, M_d, Q, A_0$, and $D_0$.  It is clear that $A, Q, D_0$ are vertices.  Further, $M_d=AP \cdot D_0Q$ by Theorem 2.1 and $A_0=AQ \cdot D_0A_0$ by Theorem 2.4.

We now show that $D_0, A_0$, and $R$ are collinear, from which we obtain $R=AP \cdot D_0A_0$.  Since $A_0E_0A_0'F_0$ joins the midpoints of the sides of the quadrilateral $F_1E_1E_3F_3$, it is a parallelogram, so the intersection of its diagonals is the point $A_0A_0' \cdot E_0F_0=N_1$.  Hence, $N_1$ bisects $A_0A_0'$ (and with $E_0F_0$ also $AD_0$).  Assuming that $P$ is an ordinary point, let $M$ be the midpoint of $PQ'$; then $K(A)=D_0, K(Q')=M$, so $AQ'$ is parallel to to $D_0M$.  Now $R$ and $M$ are midpoints of sides in triangle $AQ'P$, so $RM$ is a line through $M=K(Q')$ parallel to $AQ'$, hence we have the equality of the lines $RM=D_0M=D_0R$.  If $T=A_0'N_1.D_0R$, then triangles $AN_1A_0'$ and $D_0N_1T$ are congruent ($\angle D_0TN_1 \cong \angle AA_0'N_1$ and AAS), so $N_1$ bisects $A_0'T$ and $T=A_0$.  (Note that $N_1$, as the midpoint of $E_0F_0$, lies on $AD_0$, and $A_0$ and $A_0'$ are on opposite sides of this line; hence $N_1$ lies between $A_0$ and $A_0'$.) This shows that $D_0, R$, and $A_0$ are collinear.  By symmetry, $D_0, A_0'$, and $R'$ are collinear whenever $P'$ is ordinary.  If $P'=Q$ is infinite, use the congruence $AN_1A_0 \cong D_0N_1A_0'$ to get that $D_0A_0' || AA_0 =AQ$, which shows that $D_0, A_0'$, and $Q$ are collinear.  Thus, the last vertex of the quadrilateral
$$\Lambda'=(D_0Q')(D_0A_0')(AP')(AQ')$$
is $R'=AP' \cdot D_0A_0' =Q$ in this case.

Now consider the hexagon $AM_d'RD_0M_dQ'$ (if $P$ and $P'$ are ordinary).  Alternating vertices of this hexagon are on the lines $l=AP$ and $m=D_0Q'$, by Corollary 2.2, so the theorem of Pappus implies that intersections of opposite sides, namely,
$$AM_d' \cdot D_0M_d, \ \ AQ' \cdot RD_0, \ \ \textrm{and} \ M_dQ' \cdot M_d'R,$$
are collinear.  The point $AM_d' \cdot D_0M_d=AP' \cdot D_0Q$ is on the line at infinity because $K(AP')=D_0Q$. By the above argument, $AQ' \cdot RD_0$ is also on the line at infinity.  Hence, $M_dQ'$ is parallel to $M_d'R$.  Since $Q'M_d'$ is parallel to $AP=M_dR$ (Theorem 2.1 and its corollary), $M_dQ'M_d'R$ is a parallelogram and the intersection of the diagonals $Q'R \cdot M_dM_d'$ is the midpoint of $M_dM_d'=K(D_1D_3)$ (Corollary 2.2).  But this midpoint is $N_1=K(D_0)$, since $D_0$ is the midpoint of $D_1D_3$.  Hence, $N_1$ also bisects $Q'R$, and by symmetry, $QR'$.

We have shown that $N_1$ bisects the segments between pairs of corresponding vertices in the sets
$$\{A, R, M_d, Q, A_0, D_0\} \ \ \textrm{and} \ \ \{D_0, Q', M_d', R', A_0', A\}.$$  If $P'=Q$ is infinite, we replace $R'$ by $Q$ in the second set of vertices, and we get the same conclusion since $Q$ is then fixed by the half-turn about $N_1$.  This proves the theorem.
\end{proof}

\begin{cor}
a) If $P$ and $P'$ are ordinary, the Euclidean quadrilaterals $RA_0QM_d$ and $Q'A_0'R'M_d'$ are congruent.\\
b) If $P$ is ordinary, the points $D_0, R, A_0$, and $M=K(Q')$ are collinear, where $R$ is the midpoint of segment $AP$.  The point $M=K(Q')$ is the midpoint of segment $D_0R$.\\
c) If $P'$ is infinite, then $Q, M_d, D_0, A_0'$, and $K(A_0)$ are collinear.
\end{cor}

\begin{proof} Part a) is clear from the proof of the theorem.  For part b), we just have to prove the second assertion.  The theorem implies that quadrilateral $AQ'D_0R$ is a parallelogram, since $AQ'$ is parallel to $D_0A_0=D_0R$, $AR=AP$ is parallel to $D_0Q'$, and $R=AP \cdot D_0A_0$.  Thus, segment $AQ'$ is congruent to segment $D_0R$, and $D_0M=K(AQ')$ is half the length of $AQ' \cong D_0R$.  $M$ is clearly on the same side of line $D_0Q'$ as $P$ and $R$, so $M$ is the midpoint of $D_0R$.  Part c) follows by applying the complement map to the collinear points $P'=Q, D_3, A$, and $A_0$ and appealing to the argument in the second paragraph of the above proof.
\end{proof}
\bigskip

There are statements corresponding to Theorem 2.5 for the quadrilaterals $(BP)(BQ)(E_0Q)(E_0B_0)$ and $(CP)(CQ)(F_0Q)(F_0C_0)$. \medskip

As a second consequence of of Theorem 2.4, we prove the following theorem of Grinberg.  We give a simple proof using the cross-ratio.  For Grinberg's proof see [gr2].

\begin{thm}[{Theorem 8 in [gr1]}]  Suppose $P_1$ and $P_2$ are cyclocevian conjugates with respect to triangle $ABC$.  Then their isotomcomplements $Q_1$ and $Q_2$ are isogonal conjugates with respect to $ABC$.  Equivalently, we have
$$P_2=\phi(P_1)=\iota \circ K^{-1} \circ \gamma \circ K \circ \iota(P_1),$$
where $\iota$ is the isotomic map and $\gamma$ is the isogonal map.
\end{thm}

\begin{proof}
Let $P_1$ and $P_2$ be cyclocevian conjugates in triangle $ABC$, $D, E, F$ the traces of $P_1$, and $D',E',F'$ the traces of $P_2$, on sides $BC, AC, AB$, respectively, so that all six traces lie on a circle (see [y1] and [y3]).  Also, on side $EF$ of triangle $DEF$ let $V$ be the trace of the angle bisector of $\angle BAC = \angle FAE$.  On the one hand, the cross-ratio of the lines $AB, AC, AQ_1$, and $AV$ is given by
$$A(BC,Q_1V)=\frac{\textrm{sin} BAQ_1}{\textrm{sin}Q_1AC} \div \frac{\textrm{sin}BAV}{\textrm{sin}VAC}=\frac{\textrm{sin}BAQ_1}{\textrm{sin}Q_1AC}.$$
Next, consider the isogonal conjugate $\gamma(Q_1)=Q_1^\gamma$.  Since $\angle BAQ_1^\gamma \cong \angle CAQ_1$ and $\angle CAQ_1^\gamma \cong \angle BAQ_1$, we have
$$A(BC, Q_1^\gamma V)= \frac{\textrm{sin}BAQ_1^\gamma}{\textrm{sin}Q_1^\gamma AC}=\frac{\textrm{sin}Q_1AC}{\textrm{sin}BAQ_1}=\frac{1}{A(BC,Q_1V)}.$$
On the other hand, by Theorem 2.4 and the fact that $A_0$ is the midpoint of  $EF$ we have
$$A(BC,Q_1V) = (FE, A_0V)=\frac{FA_0}{A_0E} \div \frac{FV}{VE} = \frac{AE}{AF},$$
since $\displaystyle \frac{FV}{VE}=\frac{AF}{AE}$ in triangle $AFE$.  In the same way, we have
$$A(BC,Q_2V) = \frac{AE'}{AF'}=\frac{AF}{AE}=\frac{1}{A(BC,Q_1V)}=A(BC,Q_1^\gamma V);$$
the second equality holding because $E, E', F, F'$ lie on a circle, so that the products $AE' \cdot AE = AF' \cdot AF$ are equal.  This implies that $AQ_2$ is precisely the reflection of $AQ_1$ across the angle bisector, i.e., $AQ_2=AQ_1^\gamma$.  Applying the same argument to the vertices $B$ and $C$, we see that $Q_2$ is the isogonal conjugate of $Q_1$.
\end{proof}
\bigskip

\end{section}

\begin{section}{Cevian triangles and affine maps.}
In this section we consider the affine transformation $T_P$ which maps the triangle $ABC$ to the cevian triangle $DEF$ of point $P$, so that $T_P(A)=D$, $T_P(B)=E$, $T_P(C)=F$. We also consider some important points related to the mapping $T_P$ on the sides of $DEF$. We first give a basic lemma in order to prove geometrically that the fixed point of the affine transformation $T_P$ is $Q$.

\begin{figure}
\[\includegraphics[width=3.25in]{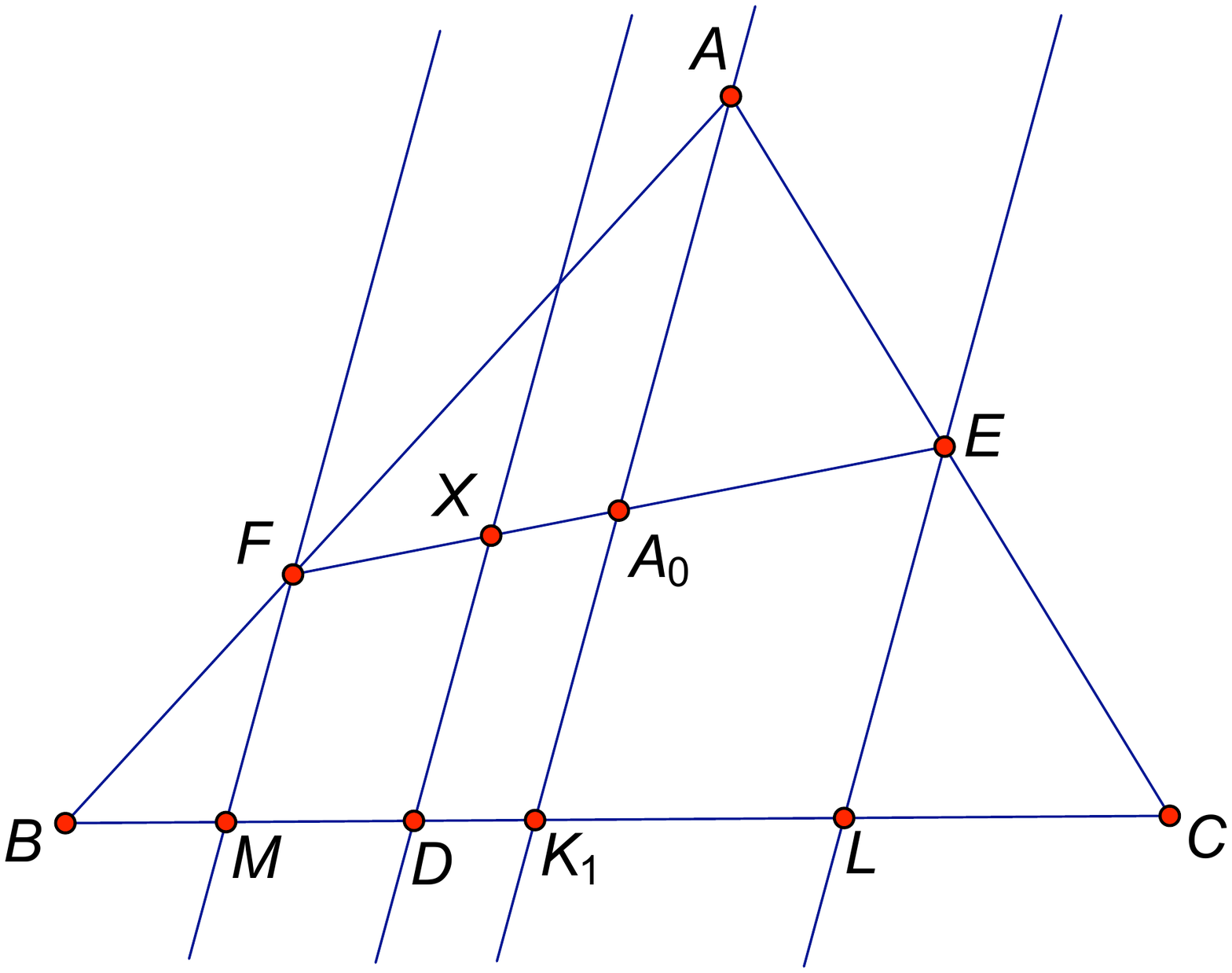}\]
\caption{Lemma \ref{lem:2.1}}
\label{fig:2.1}
\end{figure}

\begin{lem}
Let $X$ be on $EF$ such that the signed distances satisfy $\dfrac{FX}{XE} = \dfrac{BD}{DC}$.
Then $DX \pa AA_0$.
\label{lem:2.1}
\end{lem}

\begin{proof} (See Figure 4.)  Draw lines through $E$ and $F$ parallel to $AA_0$, and let them intersect $BC$ at $L$ and $M$, respectively, and let $AA_0$ intersect $BC$ at $K_1$. Draw a line through $D$ parallel to $AA_0$ and let it intersect $EF$ at $Y$. We must show that $X = Y$. The parallel lines give us the equalities
\[\frac{CE}{EA} = \frac{CL}{LK_1}, \qquad \frac{AF}{FB} = \frac{K_1M}{MB}, \qquad \frac{FY}{YE} = \frac{MD}{DL}.\]
By Ceva's theorem, $1 = \dfrac{AF}{FB}\dfrac{BD}{DC}\dfrac{CE}{EA} = \dfrac{K_1M}{MB}\dfrac{BD}{DC}\dfrac{CL}{LK_1}$. \\ \\

\noindent Since $A_0$ is the midpoint of $EF$, $K_1$ is the midpoint of $LM$, so $LK_1 = K_1M$ implies that
\[1 = \frac{CL}{MB}\frac{BD}{DC},\text{ so } \frac{BM}{LC} = \frac{MB}{CL} = \frac{BD}{DC} = \frac{BM + MD}{DL + LC} = \frac{BM + MD}{LC + DL}.\]
This last equality implies that 
\[\frac{BM}{LC} = \frac{MD}{DL} = \frac{FY}{YE}\text{, i.e. } \frac{FX}{XE} = \frac{BD}{DC} = \frac{BM}{LC} = \frac{FY}{YE}.\]

\noindent But there is exactly one point $X$ on $EF$ such that the signed ratio $\dfrac{FX}{XE}$ equals $\dfrac{BD}{DC}$, so $X = Y$.
\end{proof}

\begin{thm}[Ehrmann] If $T_P$ is the unique affine mapping which takes $ABC$ to $DEF$, then $T_P(Q) = Q$.  (This holds even when the point $P$ lies on $l_\infty$.)
\label{thm:2.2}
\end{thm}

\begin{proof} (See Figure 5.)  We show that $AA_0$ passes through $T_P(Q)$. It will follow similarly that $BB_0$ and $CC_0$ also pass through $T_P(Q)$. This implies the result because these lines intersect at $Q$.

First, if $K$ is the complement map with respect to triangle $ABC$ and $K'$ the complement map with respect to triangle $DEF$, then $T_P \circ K = K' \circ T_P$. This is because $T_P$ preserves ratios; so if $Y_1 = K(Y)$, then $Y_1$ is collinear with $G$ and $Y$, and $YG = 2\cdot GY_1$ implies $T_P(Y)T_P(G) = 2 \cdot T_P(G)T_P(Y_1)$; hence $T_P(Y_1) = K'(T_P(Y))$, since $G'=T_P(G)$ is the centroid of $DEF$.

\begin{figure}
\[\includegraphics[width=3.5in]{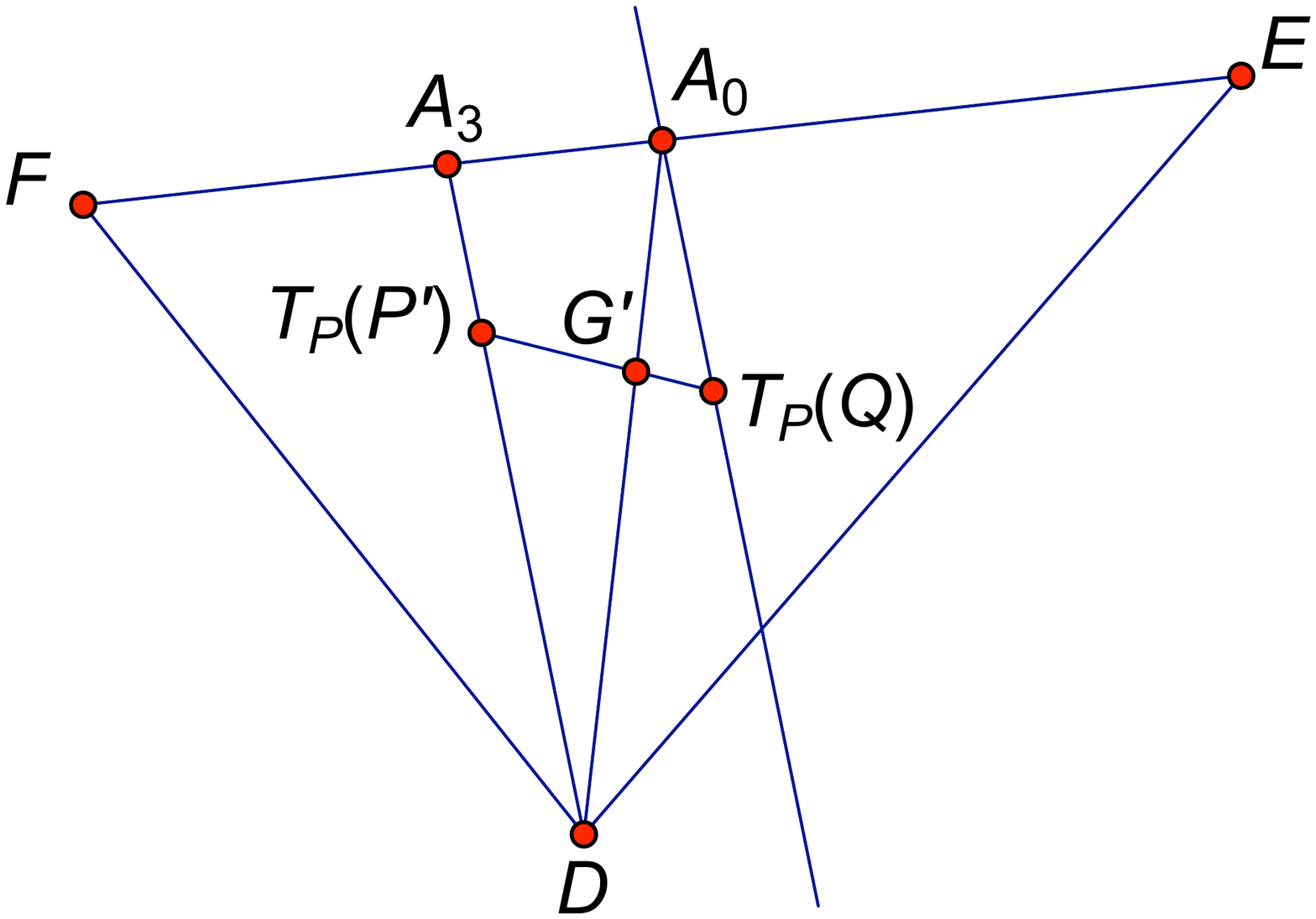}\]
\caption{$T_P(Q) = Q$}
\label{fig:2.2}
\end{figure}
Now, $T_P(Q) = T_P(K(P')) = K'(T_P(P'))$, so we really need to prove that $AA_0$ passes through the complement, in triangle $DEF$, of $T_P(P')$. Since $P'$ lies on $AD_3$, $T_P(P')$ lies on $T_P(A)T_P(D_3) = DA_3$ and
\[\frac{BD}{DC} = \frac{D_3C}{BD_3} = \frac{A_3F}{EA_3} = \frac{FA_3}{A_3E}.\]
Lemma \ref{lem:2.1} now gives that $DA_3 \pa AA_0$. If $P' = Q$ is an infinite point, then this implies that $T_P(P') = Q$, and so $T_P(Q)=K'(Q)=Q$ is fixed. If $P'$ and $Q$ are ordinary, then letting $G' = T_P(G)$ and $I = T_P(P')G'\cdot AA_0$ we have that $\Delta T_P(P')DG' \sim \Delta IA_0G'$.  Since $G'$ is the centroid of triangle $DEF$, we have $DG'=2 \cdot G'A_0$, so also $T_P(P')G' = 2 \cdot G'I$, which means that $I = T_P(Q)$ is the complement of $T_P(P')$ and $AA_0$ lies on the point $T_P(Q)$.
\end{proof}

\begin{cor}\label{cor:2.2}
The point $Q$ is the complement of $T_P(P')$ with respect to the triangle $DEF$.
\end{cor}

\begin{lem}
Let $G$ be the centroid of $ABC$; $E$ and $F$ points on $AC$ and $AB$, respectively, distinct from $A$, $B$ and $C$; and $AG \cdot EF = A^*$. Then
\[\frac{EA^*}{A^*F} = \frac{AE}{AF} \cdot \frac{AB}{AC}.\]
\label{lem:2.3}
\end{lem}

\begin{proof} Let $V$ be the trace on segment $EF$ of the angle bisector of $\angle BAC = \angle FAE$, and let $V'$ be its trace on $BC$. If $(EF,A^*V)$ denotes the cross-ratio of these four points, we have
\[\frac{EA^*}{A^*F} \div \frac{AE}{AF} = (EF, A^*V) \stackrel{A}{=} (CB, D_0V') = \frac{BV'}{V'C} = \frac{AB}{AC}.\]
\end{proof}

In order to prove the next theorem we will make use of the following involution on the line $BC$. (There are similar involutions for $AB$ and $AC$). Let $\mu$ be the perspectivity taking a point on $EF$ to a point on $BC$ by projection from $A$. We define $\pi = \mu \circ T_P$. Since $T_P$ maps $BC$ to $EF$, $\pi$ maps $BC$ to itself. Thus, if $Y$ is a point on $BC$,
\[\pi(Y) = \mu(T_P(Y)) = AT_P(Y)\cdot BC.\]
Since $\pi(B) = C$ and $\pi(C) = B$, $\pi$ interchanges two points and is therefore an involution on $BC$. The significance of this mapping is that if a point $Y$ on $BC$ maps to $T_P(Y)$ on $EF$, then $T_P$ maps the intersection $AT_P(Y)\cdot BC = Y'$ back to $AY\cdot EF = T_P(Y')$.  \bigskip

Now recall the definition of the points
\[D_0 = AG\cdot BC, \qquad D_1 = D = AP\cdot BC, \qquad D_2 = AQ\cdot BC,\]
\[D_3 = AP'\cdot BC, \qquad D_4 = AQ' \cdot BC, \qquad D_5 = AX\cdot BC.\]
\noindent Here $G$ is the centroid of $\Delta ABC$, $Q$ is the isotomcomplement of $P$, and $Q'$ is the isotomcomplement of $P'$.  Also, $X$ is defined to be the intersection of the cevians $AA_3, BB_3$, and $CC_3$, where 
\[A_3 = T_P(D_3), B_3 = T_P(E_3), C_3 = T_P(F_3).\] \smallskip
This intersection exists because $A_3B_3C_3 = T_P(D_3E_3F_3)$ is the cevian triangle for $T_P(P')$ with respect to triangle $DEF = T_P(ABC)$, and is therefore perspective to $ABC$, by the cevian nest theorem [ac, p. 165, Supp. Ex. 7].  Furthermore, $A_j=T_P(D_j)$ for $0 \le j \le 5$.

\begin{figure}
\[\includegraphics[width=4in]{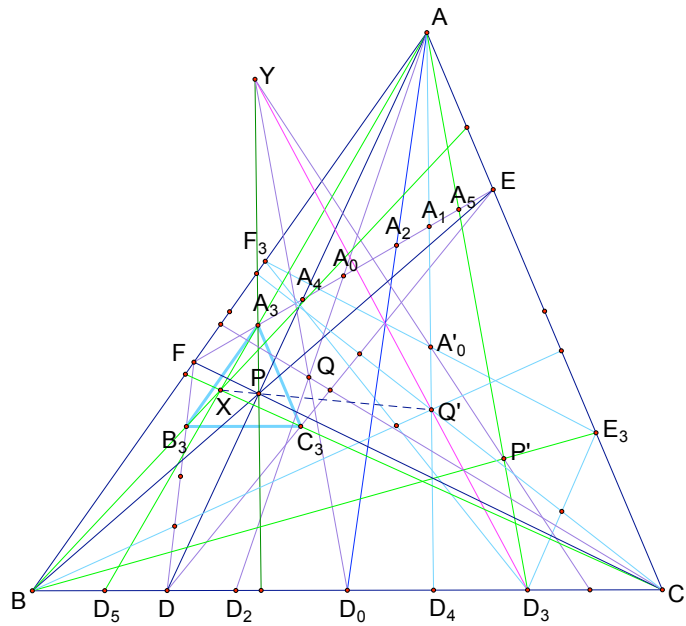}\]
\caption{The points $A_i$ and $D_i$, $i=0, 1, 2, 3, 4, 5$.}
\label{fig:2.3}
\end{figure}

\begin{thm}[Collinearity Theorem] The following sets of 4 points are collinear: $AA_0QD_2, AA_1Q'D_4, AA_2GD_0, AA_3XD_5, AA_4PD_1, AA_5P'D_3$. (Similar statements hold for the other vertices $B$ and $C$. See Figure \ref{fig:2.3} below.  This theorem also holds when $P$ or $P'$ is infinite.)
\label{thm:2.4}
\end{thm}

\begin{proof} (Se Figure 6.)  The collinearity of the four points $AA_3XD_5$ is immediate from the definition of the point $X$. Hence $\pi(D_3) = D_5$. Now $\pi(D_5) = D_3$ implies that $AP' \cdot EF = T_P(D_5) = A_5$ and so $AA_5P'D_3$ is a collinear set.\\
\\
The collinearity $AA_0QD_2$ is immediate from Theorem \ref{thm:1.3}, where we note that $T_P(D_0) = A_0$ since $T_P$ preserves ratios along lines. Hence $\pi(D_0) = D_2$, which implies that $\pi(D_2) = D_0$, so $AA_2GD_0$ is also a collinear set, where $A_2 = T_P(D_2)$.\\
\\
It remains to prove that the sets $AA_1Q'D_4$ and $AA_4PD_1$ are collinear. To do this we first redefine $A_1$ as the intersection $A_1 = AQ'\cdot EF$ and we show that $A_1 = T_P(D_1)$. This will imply the collinearity of the points $AA_4PD_1$ with $A_4 = T_P(D_4)$ using the map $\pi$. Using the cross-ratio and the fact that $A_0$ is the midpoint of segment $EF$ we have that
\[\frac{EA_1}{A_1F} = \frac{A_1E}{FA_1} = (FE, A_0A_1) \stackrel{A}{=}(BC, D_2D_4) = \frac{BD_2}{D_2C} \div \frac{BD_4}{D_4C}.\]
Using Lemma 3.4 with $A^* = A_2$ on $EF$ and $A^* = A'_2$ on $E_3F_3$, where we denote the analogues of the points $D_i, A_i$ corresponding to $P'$ by $D_i', A_i'$ (so that $D_2'=D_4$ and $T_{P'}(B)=E_3$, etc.), we have
\[\frac{BD_2}{D_2C} = \frac{EA_2}{A_2F} = \frac{AE}{AF}\cdot\frac{AB}{AC}\text{ and } \frac{BD_4}{D_4C} = \frac{BD'_2}{D'_2C} = \frac{E_3A'_2}{A'_2F_3} = \frac{AE_3}{AF_3}\cdot \frac{AB}{AC}.\]
Hence, the above ratio $EA_1/A_1F$ becomes
\[\frac{EA_1}{A_1F} = \frac{AE}{AF}\frac{AF_3}{AE_3} = \frac{AE}{AF} \frac{FB}{EC} = \frac{EA}{AF} \frac{FB}{CE} = \frac{BD}{DC},\]
by Ceva's theorem and the fact that $(F, F_3)$ and $(E, E_3)$ are isotomic pairs of conjugates on $AB$ and $AC$, respectively. This implies that $A_1 = T_P(D) = T_P(D_1)$ and completes the proof of the theorem.
\end{proof}

\noindent The following fact is a simple corollary of Theorem \ref{thm:2.4}, but is important enough in the following development to state as a theorem.

\begin{thm} 
$T_PK(P) = T_P(Q') = P$.
\label{thm:2.5}
\end{thm}

\begin{proof}
$T_P(Q') = T_P(AD_4\cdot BE_4) = DA_4 \cdot EB_4 = DA \cdot EB = P$.
\end{proof}

\bigskip

\begin{thm}[Homothety theorem]
\label{thm:2.6}
The affine mapping $T_PT_{P'}$ taking $ABC$ to $A_3B_3C_3$ is either a homothety, whose center is the ordinary point $X = AA_3\cdot BB_3 = AA_3 \cdot CC_3$ lying on the line $PQ'$, or a translation in the direction of the line $PQ'$. Thus, triangles $ABC$ and $A_3B_3C_3$ are either homothetic or congruent.
\end{thm}

\begin{proof}
Write $T_1$ for $T_P$ and $T_2$ for $T_{P'}$, and let $\li$ be the line at infinity, as usual. We first show that $T_1$ and $T_2$ are inverse mappings on $\li$. Assume that the points $P'$ and $Q$ are ordinary. Note that $T_2$ maps the line $AQ$ to $D_3P' = AP'$, since $T_2(Q) = P'$ (by Theorem 3.6). Hence, $T_2$ maps the point at infinity $A_\infty$ on $AQ$ to the point at infinity $D_\infty$ on $AP'$. On the other hand, $AP'$ is parallel to $D_0Q$ (Corollary \ref{cor:1.1}), so $D_\infty$ lies on $D_0Q$. Moreover, $T_1(D_0Q) = A_0Q = AQ$ by Theorem \ref{thm:2.2} and Theorem \ref{thm:1.3}. Therefore, $T_1(D_\infty) = A_\infty$. Arguing in the same way with $B_\infty = BQ\cdot l_\infty, C_\infty = CQ\cdot \l_\infty$ and $E_\infty = BP'\cdot \li, F_\infty = CP'\cdot \li$, we see that on the line $\li$
\begin{center}
$T_2$ induces the projectivity $A_\infty B_\infty C_\infty \barwedge D_\infty E_\infty F_\infty$, while\\
$T_1$ induces the projectivity $D_\infty E_\infty F_\infty \barwedge A_\infty B_\infty C_\infty$.
\end{center}
The fundamental theorem of projective geometry now implies that $T_1T_2$ is the identity map on $\li$. (Note that $A_\infty, B_\infty, C_\infty$ are distinct points since they lie on the concurrent lines $AQ, BQ, CQ$.  If $Q$ were on $AB$, say, then $P'$ would lie on the anti-complementary triangle of $ABC$, so the trace $F_3=CP' \cdot AB$ of $P'$ would lie on $l_\infty$, implying that $F=F_3$, and $P$ would also lie on the anti-complementary triangle, contrary to the standing hypothesis about $P$.) On the other hand, if $P' = Q$ is an infinite point, then $P$ and $Q'$ are ordinary, and we can apply the above argument to the map $T_2T_1$. Since this map is the identity on $\li$, so is $T_1T_2$.\\
\\
Now it is clear that $T_1T_2(ABC) = T_1(D_3E_3F_3) = A_3B_3C_3$. By the above argument, the mapping $\cS = T_1T_2$ fixes the point $AA_3\cdot \li$, so $\cS(A) = A_3$ implies that $AA_3$ is an invariant line, as are $BB_3$ and $CC_3$ (and any line of the form $Y\cS(Y)$). Hence $X = AA_3\cdot BB_3$ is an invariant point of $\cS$. If $X$ is an ordinary point, then $\cS$ is a projective homology [co2], which must be a homothety since it takes any line to a parallel line. It follows that $ABC$ and $A_3B_3C_3$ are homothetic from the center $X$. Since $\cS(Q') = T_1T_2(Q') = T_1(Q') = P$, $X$ lies on the line $PQ'$.\\
\\
If $X$ is a point at infinity, then $AA_3, BB_3, CC_3$ are parallel, so $\cS$ must be a translation. Since $\cS(Q') = P$, the translation is in the direction of the line $PQ'$.
\end{proof}

In order to further describe the point $X$ we prove the following theorem of Grinberg.

\begin{thm}[Thm. 4 in \cite{gr1}]\label{thm:2.7} If parallel lines are drawn to the sides $EF, DF, DE$ of the cevian triangle for $P$ through the respective vertices $A, B$, and $C$, then the resulting triangle is the anticevian triangle of $Q$ for triangle $ABC$.
\end{thm}

\begin{proof}
Consider the polarity induced by the inscribed conic $\cI$ in $ABC$ through the points $D, E, F$. Then the lines $d = BC, e = AC, f = AB$ are the polars of the points $D, E, F$, while the polars of $A, B, C$ are the lines $a = EF, b = DF, c = DE$. The lines through $A, B, C$ parallel to these lines are the polars $a_0, b_0, c_0$ of the midpoints $A_0, B_0, C_0$ of the sides of $DEF$. This is because the polar of $A_0$ is the line through $A$ and the harmonic conjugate of $A_0$ with respect to $E, F$, which is the point at infinity on $EF$. It follows from this that the pole of the line at infinity lies on $AA_0$, so Theorem \ref{thm:1.3} implies that this pole is $Q$. Let the vertices of the triangle formed by $a_0, b_0, c_0$ be $A' = b_0\cdot c_0, B' = a_0\cdot c_0, C' = a_0\cdot b_0$. We must show that $A'Q$ lies on $A$, $B'Q$ lies on $B$, $C'Q$ lies on $C$. Using the polarity we see that $A'Q$ lies on $A$ if and only if $B_0C_0\cdot q=B_0C_0 \cdot \li$ lies on $EF$. But this is obvious because $B_0$ and $C_0$ are midpoints in triangle $DEF$, so that $B_0C_0$ is parallel to $EF$.
\end{proof}

We can now prove

\begin{thm}\label{thm:2.8} The fixed point $X = AA_3\cdot BB_3$ of $T_PT_{P'}$ is the $P$-ceva conjugate of $Q$. The cevian triangle of $P$ is homothetic to the anticevian triangle of $Q$ for triangle $ABC$ from the center $X$ if $X$ is an ordinary point, and is congruent to this triangle otherwise.
\end{thm}

\begin{proof}
Let $\cS = T_PT_{P'}$. Consider the triangle $A'B'C' = \cS^{-1}(DEF)$. Then the sides of $A'B'C'$ are parallel to the sides of $DEF$, since $\cS$ fixes points at infinity. Furthermore, triangle $A_3B_3C_3$ is inscribed in triangle $DEF$, so $\cS^{-1}(A_3B_3C_3) = ABC$ is inscribed in triangle $A'B'C'$. By Theorem \ref{thm:2.7}, $A'B'C'$ must be the anticevian triangle of $Q$. The $P$-ceva conjugate of $Q$ is by definition the perspector of $DEF$ and $A'B'C'$, and by construction this point is the center $X = AA_3\cdot BB_3$ of Theorem 3.7, whether $\cS$ is a homothety or a translation. This proves the assertion.
\end{proof}

\begin{cor}\label{cor:2.8}
\begin{enumerate}[(a)]
\item The triangle $T_{P'}^{-1}(ABC)$ is the anticevian triangle of $Q$ for $ABC$.
\item If $X'$ is the $X$-point corresponding to $P'$, then $T_P(X') = X$ and $T_{P'}(X) = X'$.
\item $X$ is an ordinary point if and only if $X'$ is.
\end{enumerate}
\end{cor}

\begin{proof}
(a) The anticevian triangle of $Q$ is $A'B'C' = \cS^{-1}(DEF) = T_{P'}^{-1}(ABC)$. (b) The perspector of $D_3E_3F_3$, the cevian triangle of $P'$, and $T_P^{-1}(ABC)$, the anticevian triangle of $Q'$, is $X'$. It follows that $T_P(X')$ is the perspector of triangles $A_3B_3C_3$ and $ABC$, hence $T_P(X') = X$. The second assertion in (b) follows on switching $P$ and $P'$. Part (c) is immediate from (b).
\end{proof}

\begin{thm}\label{thm:2.9} If $P$ does not lie on $\iota(\li)$ (the Steiner circumellipse for $ABC$), then $Q$ is the only fixed point of $T_P$ in the finite plane.
\end{thm}

\noindent {\bf Remark.} If the point $P$ does lie on the Steiner circumellipse for $ABC$, then it can be shown that $T_P$ has no ordinary fixed points, but does have the line $GT_P(G)$ as a fixed line, where $G$ is the centroid of $ABC$. See the proofs of Lemma 2.5 and Theorems 2.4 and 4.3 in [mm].

\begin{proof} Note, since $P$ does not lie on $\iota(\li)$, that the points $P'$ and $Q$ are ordinary points.  We already know from Theorem 3.2 that $T_P$ fixes $Q$. Suppose there is another finite fixed point $R$ of $T_P$. Then $m = QR$ is an invariant line for $T_P$. The line at infinity, $\li$, is also an invariant line since $T_P$ is an affine transformation. Therefore, $T_P$ fixes the point $M_\infty = m\cdot \li$. Since $T_P$ fixes three points on $m$, it fixes every point on $m.$\\
\\
Suppose $m\cdot BC = S$. Then $S = T_P(S) = T_P(m\cdot BC) = m\cdot EF$, which implies $S = BC\cdot EF$. Similarly, $m\cdot AB = AB\cdot DE$ and $m\cdot AC = AC\cdot DF$, so $m$ is the line of perspectivity of triangles $DEF$ and $ABC$. Hence, $m$ is the trilinear polar of the point $P$, and $S$ is the harmonic conjugate of $D$ with respect to $B$ and $C$. Projecting line $BC$ to line $FE$ from $A$ gives $(BC, DS) = -1 = (FE, A_4S) = (EF, A_4S)$ and since $S = T_P(S)$, the signed ratio of $S$ along $BC$ is the same as its ratio along $EF = T_P(BC)$:
\[\frac{BD}{DC} = -\frac{BS}{SC} = -\frac{ES}{SF} = \frac{EA_4}{A_4F}.\]
But the only point $A^*$ on $EF$ such that $\dfrac{BD}{DC} = \dfrac{EA^*}{A^*F}$ is $A_1 = T_P(D_1)$, so $A_1 = A_4$. By Theorem \ref{thm:2.4}, $AA_1 = AQ'$ and $AA_4 = AP$, so we have $AQ' = AP$. Since $K(P) = Q'$, this implies that the centroid $G$ lies on $AP$, so $P$ is on $AG$. Similarly, $P$ is on $BG$ and $CG$, so $P = G$. But then $T_P = T_G = K$ and the line of perspectivity $m = QR$ is the line at infinity, yet $Q = G$ is not on $m=\li$: a contradiction.
\end{proof}

\begin{thm}\label{thm:2.10} If $P$ is ordinary, the point $Q'=K(P)$ is the isotomcomplement of $Q$ with respect to the anticevian triangle of $Q$ for $ABC$.
\end{thm}

\begin{proof}
The unique affine mapping taking the vertices of the anticevian triangle $A'B'C' = T_{P'}^{-1}(ABC)$ of $Q$ to the vertices of $ABC$ is $T_{P'}$. Since the point $P$ is ordinary, the point $P'$ does not lie on the Steiner circumellipse for $ABC$, and therefore the point $Q = T_{P'}^{-1}(P')$ does not lie on the Steiner circumellipse for $A'B'C'$.  (If $\iota'$ is the isotomic map for $A'B'C'$, then $T_{P'}^{-1} \circ \iota = \iota' \circ T_{P'}^{-1}$, and affine maps fix the line $l_\infty$, so the Steiner circumellipse for $ABC$ is mapped to the Steiner circumellipse for $A'B'C'$.) It follows that the isotomcomplement of $Q$ with respect to $A'B'C'$ is the unique ordinary fixed point of the mapping $T_{P'}$, by Theorem \ref{thm:2.9}, so this point must be $Q'$.
\end{proof}

Next, we characterize the points $P$ on the Steiner circumellipse in terms of the mapping $T_P$.

\begin{thm}\label{thm:2.11} The point $P$ ($\ne A, B$, or $C$) lies on the Steiner circumellipse $\iota(\li)$ of $ABC$ if and only if $T_PT_{P'} = K^{-1}$.
\end{thm}

\begin{proof}
Assume $P$ lies on $\iota(\li)$, so that the point $P' = Q$ is an infinite point. Let $A', B', C'$ be the midpoints of segments $AD_3, BE_3, CF_3$. We claim that $A'B'C'$ is the anticevian triangle of $Q$ with respect to $ABC$. Note that $A' = M'_d, B' = M'_e, C' = M'_f$ in the notation of Theorem \ref{thm:1.1}. By the corollary to that theorem, $K(DEF) = A'B'C'$. Thus, the sides of $A'B'C'$ are parallel to the sides of $DEF$. By Grinberg's Theorem 3.8 we just have to show that $ABC$ is inscribed in $A'B'C'$. We note that the complete quadrangle $ABCP'$ has the diagonal triangle $D_3E_3F_3$. By the Collinearity Theorem (\ref{thm:2.4}) we know that $AA'_4P'D_3$ is a collinear set of points, where $A'_4$ lies on $E_3F_3$. Thus, by the definition of a harmonic set, $A$ and $P'$ are harmonic conjugates with respect to $A'_4$ and $D_3$.  This implies that $A$ is the midpoint of $D_3A'_4$. Now $B'$ and $C'$ are the midpoints of $BE_3$ and $CF_3$. Since $B, C$, and $D_3$ are collinear, as are $E_3, F_3$, and $A'_4$, and furthermore the lines $BE_3, CF_3$, and $D_3A'_4 = AP'$ are parallel, it is clear that the respective midpoints $B', C'$, and $A$ are collinear as well. Arguing the same with the other vertices shows that $ABC$ is inscribed in $A'B'C'$. Hence, $A'B'C' = K(DEF) = KT_P(ABC)$ is the anticevian triangle of $Q$. From Corollary \ref{cor:2.8} we deduce that $KT_P(ABC) = T_{P'}^{-1}(ABC)$, whence the desired equation $T_PT_{P'} = K^{-1}$ follows.\\
\\
Conversely, suppose that $T_PT_{P'} = K^{-1}$. Then $T_P(D_3E_3F_3) = A_3B_3C_3 = K^{-1}(ABC)$ is the anticomplementary triangle of $ABC$, from which it is clear that $AA_3, BB_3$, and $CC_3$ all pass through the centroid $G$ of $ABC$. Theorem \ref{thm:2.4} implies that $A_3B_3C_3 = A_2B_2C_2$, hence $D_3E_3F_3 = D_2E_2F_2$, which gives that $P' = Q$. Hence, the point $P'$ coincides with its complement and must be infinite (since $P$ cannot be $G$), i.e., $P$ lies on $\iota(\li)$.
\end{proof}

\begin{cor}\label{cor:2.11} If $P$ lies on the Steiner circumellipse $\iota(\li)$, then the triangle $A_2B_2C_2 = A_3B_3C_3$ is the anticomplementary triangle of $ABC$; the anticevian triangle of $Q$ with respect to $ABC$ is the triangle $K(DEF)$; and the anticevian triangle of $Q'$ is the triangle $A'_0B'_0C'_0$.
\end{cor}

\end{section}

\noindent Dept. of Mathematics, Carney Hall\\
Boston College\\
140 Commonwealth Ave., Chestnut Hill, Massachusetts, 02467-3806\\
{\it e-mail}: igor.minevich@bc.edu
\bigskip

\noindent Dept. of Mathematical Sciences\\
Indiana University - Purdue University at Indianapolis (IUPUI)\\
402 N. Blackford St., Indianapolis, Indiana, 46202\\
{\it e-mail}: pmorton@math.iupui.edu

\end{document}